\documentclass{amsart}

\usepackage{a4wide,amsmath,amssymb}
\usepackage[toc,page]{appendix}

\usepackage{url}
\usepackage{hyperref}

\usepackage{mathtools}

\usepackage{xcolor}

\newtheorem{thm}{Theorem}[section]

\theoremstyle{remark}
\newtheorem{rem}[thm]{Remark}

\theoremstyle{definition}
\newtheorem{defn}[thm]{Definition}
\newtheoremstyle{Claim}{}{}{\itshape}{}{\itshape\bfseries}{:}{ }{#1}
\theoremstyle{Claim}

\newcommand{\R}{\mathbb{R}}

\newcommand{\eps}{\varepsilon}

\theoremstyle{plain}

\makeatletter
\@namedef{subjclassname@2020}{\textup{2020} Mathematics Subject Classification}
\makeatother
\begin{document}

\title[]{On the smoothness of solutions of fully nonlinear second order equations in the plane}

\author{Alessandro Goffi}
\address{Dipartimento di Matematica e Informatica ``Ulisse Dini'', Universit\`a degli Studi di Firenze, 
viale G. Morgagni 67/a, 50134 Firenze (Italy)}
\curraddr{}
\email{alessandro.goffi@unifi.it}

\subjclass[2020]{35B65}

\keywords{Fully nonlinear equations, Evans-Krylov theorems.}
 \thanks{The author is member of the Gruppo Nazionale per l’Analisi Matematica, la Probabilit\`a e le
loro Applicazioni (GNAMPA) of the Istituto Nazionale di Alta Matematica (INdAM). He has been supported by the INdAM-GNAMPA Project 2025 ``Stabilit\`a e confronto per equazioni di trasporto/diffusione e applicazioni a EDP non lineari''. The author wishes to thank Andrea Colesanti for the reference \cite{Talenti} and Matteo Focardi for discussions on the subject of the paper. 
}

\date{\today}

\begin{abstract}
We study interior $C^{2,\alpha}$ regularity estimates for solutions of fully nonlinear uniformly elliptic equations of the general form $F(D^2u)=0$ in two independent variables and without any geometric condition on $F$. By means of the theory of divergence form equations we prove that $C^2$ solutions of the previous equation are $C^{2,\bar\alpha(\lambda/\Lambda)}$ in the interior of the domain, where $0<\lambda\leq\Lambda$ are the ellipticity constants. We finally exploit the theory of nondivergence equations in the plane to obtain $C^{2,\tilde\alpha}$ regularity for an explicit exponent $\tilde\alpha=\tilde\alpha(\lambda/\Lambda)>\lambda/\Lambda$. 
\end{abstract}

\maketitle

\section{Introduction}
In this paper we study maximal $C^\alpha$-regularity (i.e. $C^{2,\alpha}$ estimates) for solutions of second order fully nonlinear uniformly elliptic equations, with ellipticity constants $0<\lambda\leq\Lambda$, of the form
\begin{equation}\label{eqintro}
F(D^2u)=0\text{ in }\R^n
\end{equation}
in the special case $n=2$, without any geometric requirement. We are interested in the regularizing effect of Evans-Krylov type
\[
C^{2}\rightarrow C^{2,\alpha}.
\]
Our main goal is to exhibit the explicit dependence of the universal exponent $\alpha\in(0,1)$ with respect to the ellipticity constants.
The following results are well-known:
\begin{itemize}
\item If $n=2$, $F$ is continuous and uniformly elliptic, a result by L. Nirenberg \cite{Nirenberg} shows that $u\in C^{2,\alpha}$ for some small universal $\alpha\in(0,1)$ depending on $\lambda/\Lambda$, see also \cite[Remark 2]{CaffarelliYuan} and \cite[Chapter 4]{FRRO}. If $F\in C^\infty$, then $u\in C^\infty$.
\item If $n\geq 5$ there are counterexamples to the $C^{1,1}$ regularity \cite{NTV}, and the best regularity is $C^{1,\alpha}$ by fundamental results proved by L. Caffarelli \cite{C89,CC} and N. Trudinger \cite{Trudinger88}. The cases $n=3,4$ remain at this stage the main open problems in the theory: solutions are known to be $C^{1,\alpha}$ under the sole uniform ellipticity, but we do not know whether they satisfy better regularity, cf. \cite{NTV}.
\item If $F$ is uniformly elliptic and either concave or convex (i.e. Bellman equations), solutions are known to be $C^{2,\alpha}$, $\alpha\in(0,1)$ small, by a result of L.C. Evans \cite{Evans82} and N.V. Krylov \cite{Krylov82}. This was earlier proved through the theory of variational inequalities by H. Br\'ezis and L.C. Evans \cite{BrezisEvans} for the maximum of two operators. If $F$ is only continuous the maximal regularity of Bellman equations is one of the major open problems. When $n=2$ this was solved in the case $F(D^2u)=\min\{L_1u,L_2u\}$ with $L_1,L_2$ having constant coefficients, and solutions of these special Bellman PDEs are $C^{2,1}$, as proved by L. Caffarelli-D. De Silva-O. Savin \cite{CDSS}. Higher smoothness than $C^{2,\alpha}$ ($\alpha$ implicit) remains unsettled when $n\geq3$.
\item Other $C^{2,\alpha}$ results are known under smallness conditions. O. Savin \cite{Savin} proved it in the case of flat solutions, while earlier results by L.C. Evans and P.L. Lions \cite{EvansLions} provided smoothness of solutions when the equation has a large zero-th order coefficient. More recent advances are known under Cordes-type conditions \cite{HuangAIHP}.
\item X. Cabr\'e and L. Caffarelli raised the question about $C^{2,\alpha}$ smoothness when $F$ is weaker than concave/convex. There are few results in this direction, see e.g. \cite{CaffarelliYuan,CCjmpa,HuangAIHP,Collins,KrylovBookNew,GoffiJMPA} and references therein. All these results are known for a small $\alpha$.
\end{itemize}
H. Br\'ezis \cite[Comments and open questions- (2)]{Brezis} and H. Br\'ezis-L.C. Evans \cite[Remark 4.2]{BrezisEvans} asked whether one can expect better regularity than $C^{2,\alpha}$ for $\alpha$ small and $W^{3,p}$ for some $p$. The H\"older exponent is implicit since one usually applies the De Giorgi-Nash-Moser or Krylov-Safonov theory to the equation satisfied by the second derivatives. $C^{2,\alpha}$ smoothness is false when $n\geq 5$ by the counterexamples in \cite{NTV}, but it remains an interesting question in the lower dimensional cases $n\leq4$, starting from the plane $\R^2$, and for convex equations. We are only aware of a result by Q. Huang \cite{HuangProcAMS} showing that if $F$ is uniformly elliptic, then $u\in C^{2,\alpha}$ for any $\alpha\in(0,1)$ and any $n\geq2$, with estimates depending on the first derivatives of $F$, whenever the equation has a polynomial Liouville property. In the special case $F(D^2u)=\min\{L_1u,L_2u\}$ an optimal $C^{2,1}$ estimate is given in \cite{CDSS}. \\
As far as $W^{3,p}$ estimates are concerned, these are known for solutions of \eqref{eqintro} in any dimension for an implicit $p<1$ depending on $n$ and $\lambda/\Lambda$, as a consequence of \cite[Theorem 1.4]{NamLe} or \cite{ASS}. It is enough to apply the $W^{2,\eps}$-estimate of \cite{EvansTAMS,CC,Lin} to the equation solved by the directional derivative $\partial_e u$, for which estimates of the decay of the integrability are available \cite{ASS,NamLe,Mooney,Nascimento}. The authors in \cite{ASS} observed that $\eps=\frac{2\lambda}{\Lambda+\lambda}$ is the universal upper bound in the two-dimensional case for solutions of $\mathcal{M}^+_{\lambda,\Lambda}(D^2u)\geq0$. However, this exponent can be improved when $\lambda=\Lambda$, where $\eps=1$. The recent manuscript \cite{NTjde} addressed lower bounds for the exponent $\eps$, proving in particular the new estimate $\varepsilon > \frac{1.629 \lambda}{\lambda+\Lambda}$.\\

Our paper gives some quantitative results in the first direction, proving some $C^{2,\alpha}$ estimates and providing an explicit lower bound for the H\"older exponent depending on $\lambda/\Lambda$. This can be considered as a counterpart in $C^{2,\alpha}$ spaces of the results recently obtained in \cite{NTjde} for $W^{2,\eps}$ estimates. We use the linear theory of equations in divergence and nondivergence form and exploit some known regularity properties of their solutions that depend only on the bounds of the coefficients, but not on their regularity properties.  To our knowledge these results have not appeared in print anywhere. We first prove that, in the general case, $C^2$ solutions of \eqref{eqintro} when $n=2$ and $F$ is uniformly elliptic are more regular and satisfy
\[
u\in C^{2,\bar\alpha}.
\]
with $\bar\alpha$ explicit and depending on $\lambda$ and $\Lambda$. These results use that the second derivatives solve a suitable equation in divergence form, see Section \ref{sec;recall}. We then improve this bound via the theory of equations in nondivergence form, proving
\[
u\in C^{2,\tilde\alpha}
\]
where
\[
\tilde\alpha>\frac{\lambda}{\Lambda}.
\]
The results apply to general uniformly elliptic Bellman-Isaacs equations and even to Bellman equations given by the minimum/maximum of a family of linear operators, for which a general result is not known, except for \cite{CDSS} where $F$ is the minimum of two operators in the plane.\\
It is worth mentioning that there is a counterpart of this theory for nonlinear equations in divegence form, whose main model is the $p$-Laplacian. An explicit H\"older exponent, depending on $p$, for the regularity of the gradient of solutions of the $p$-Laplace equation was proved in \cite{ManfrediIwaniec}: solutions of $p$-harmonic functions in the plane are $C^{1,\alpha}$ with
\[
\alpha=\frac16\left(\frac{p}{p-1}+\sqrt{1+\frac{14}{p-1}+\frac{1}{(p-1)^2}}\right).
\] 
\smallskip

\textit{Outline}. In Section \ref{sec;prel} we recall the notion of uniform ellipticity for \eqref{eqintro}. Section \ref{sec;recall} provides a short survey of regularity properties of solutions to equations in divergence and nondivergence form in the plane. Section \ref{sec;div} exploits the result for equations in divergence form to prove $C^{2,\alpha}$ and $W^{3,p}$ estimates for fully nonlinear equations, while the main result in Section \ref{sec;nondiv} uses the nonvariational regularity theory in the plane to prove $C^{2,\alpha}$ estimates for a different value of $\alpha$.
\section{Preliminaries on fully nonlinear equations}\label{sec;prel}
We will consider equations of the form \eqref{eqintro} where $F:\mathrm{Sym}_n\to\R$, where $\mathrm{Sym}_n$ is the space of $n\times n$ symmetric matrices. We will assume that $F$ is continuous and uniformly elliptic, as described by the next definition.
\begin{defn}
$F$ is uniformly elliptic with ellipticity constants $0<\lambda\leq\Lambda$ if for every symmetric matrices $M,N\in\mathrm{Sym}_n$, $N\geq0$, we have
\[
\lambda\mathrm{Tr}(N)\leq F(M+N)-F(M)\leq\Lambda\mathrm{Tr}(N).
\]
\end{defn}
Extending $F$ as a function $F:\R^{n\times n}\to\R$ it is possible to define
\[
F_{ij}(M)=\frac{\partial F}{\partial m_{ij}}(M).
\]
This holds even if $F$ is not $C^1$, since the uniform ellipticity implies the Lipschitz regularity in the matrix entry: in this case the derivatives must be understood almost everywhere. If $F$ is uniformly elliptic according to the previous definition, then the linearization of $F$ is uniformly elliptic with the same ellipticity constants, namely for all $M$
\[
\lambda \mathbb{I}_n\leq F_M(M)\leq\Lambda \mathbb{I}_n,
\]
where $F_M(M)=\left(\frac{\partial F}{\partial m_{ij}}(M)\right)_{ij}$. This can be checked via the previous definition considering, for $M\in\mathrm{Sym}_n$ the function $G(t)=F(M+tN)$, $t>0$, and taking $N=\xi\otimes \xi$, $\xi\in\R^n$, which is nonnegative semidefinite. 
\section{Results for linear equations with bounded and measurable coefficients: a brief survey}\label{sec;recall}
\subsection{Divergence form equations}
We first recall the following important result due to De Giorgi-Nash-Moser about the H\"older regularity of solutions of divergence form equations with bounded and measurable coefficients.
\begin{thm}[De Giorgi-Nash-Moser]
Let $u\in H^1(B_1)$ be a solution of $\mathrm{div}(A(x)Du)=0$ in $B_1\subset\R^n$ with $A$ measurable and such that $\lambda \mathbb{I}_n\leq A\leq\Lambda \mathbb{I}_n$, $0<\lambda\leq\Lambda$. Then there exist constants $\alpha\in(0,1)$ and $C>1$ universal such that
\[
\|u\|_{C^{0,\alpha}(B_\frac12)}\leq C\|u\|_{L^2(B_1)}.
\]
\end{thm}
This result can be made quantitative for equations posed on $\R^2$. When $A$ is symmetric, an explicit exponent $\alpha=\frac{\lambda}{2\Lambda}$ was found by C. Morrey \cite{MorreyBook}, while \cite{Widman} provided the exponent $\alpha=\frac{1}{8}\sqrt{\frac{\lambda}{\Lambda}}$. The sharp result is contained in the following:
\begin{thm}[Piccinini-Spagnolo]\label{picspa1}
Let $u\in H^1(B_1)$ be a solution of $\mathrm{div}(A(x)Du)=0$ in $B_1\subset\R^2$ with $A\in\mathrm{Sym}_2$ measurable and such that $\lambda \mathbb{I}_2\leq A\leq\Lambda \mathbb{I}_2$, $0<\lambda\leq\Lambda$. Then there exists a constant $C>1$ universal such that
\[
\|u\|_{C^{0,\sqrt{\lambda/\Lambda}}(B_\frac12)}\leq C\|u\|_{L^2(B_1)}.
\]
\end{thm}
We also have the following generalization valid for certain non-symmetric matrices:
\begin{thm}[Treskunov]\label{treskunov}
Let $u\in H^1(B_1)$ be a solution of $\mathrm{div}(A(x)Du)=0$ in $B_1\subset\R^2$ with
\[
A=\begin{pmatrix}
a(x) & b(x)+d(x)\\
b(x)-d(x) & c(x)
\end{pmatrix},
\]
where $a,b,c,d$ are bounded and measurable, $|d|\leq \eta$ and for all $\xi\in\R^2$
\[
\nu|\xi|^2\leq a\xi_1^2+2b\xi_1\xi_2+c\xi_2^2\leq |\xi|^2,\ \nu\leq1.
\]
Then $u\in C^{0,\frac{2}{\pi}\sqrt{\nu}\arctan\left(\frac{\sqrt{\nu}}{\eta}\right)}(B_\frac12)$. When $\eta=0$ (hence $d=0$), we have $u\in C^{0,\sqrt{\nu}}(B_\frac12)$, which is the exponent found in Theorem \ref{picspa1} (after normalization of the ellipticity constants), and
\[
\frac{2}{\pi}\sqrt{\nu}\arctan\left(\frac{\sqrt{\nu}}{\eta}\right)<\sqrt{\nu}.
\]
\end{thm}
The H\"older exponents found in \cite{PiccininiSpagnolo} and \cite{Treskunov87} are optimal.\\

A well-known result of N. Meyers provides $L^{2+\eps}$ integrability of the gradient of $H^1$ solutions of equations in divergence form with bounded and measurable coefficients in any dimension. In the case $n=2$ and the matrix is symmetric the exponent is explicit, and we have the following \cite{LeonettiNesi}
\begin{thm}[Astala-Leonetti-Nesi]
Let $u\in H^1(B_1)$ be a solution of $\mathrm{div}(A(x)Du)=0$ in $B_1\subset\R^2$with $A\in\mathrm{Sym}_2$, measurable and such that $\frac{1}{K} \mathbb{I}_2\leq A\leq K\mathbb{I}_2$, $K\geq1$. Then $u\in L^{\frac{2K}{K-1},\infty}_{\mathrm{loc}}(B_1)$, the Marcinkiewicz space of order $\frac{2K}{K-1}$. This implies that $u\in W^{1,p}_{\mathrm{loc}}$, $p<\frac{2K}{K-1}$.
\end{thm}

For divergence equations driven by diagonal matrices $A(x)=\sigma(x)\mathbb{I}_2$ we have the following improvement of Theorem \ref{picspa1}:
\begin{thm}[Piccinini-Spagnolo]
Let $u\in H^1(B_1)$ be a solution of $\sum_{i=1}^2\partial_i(\sigma(x)\partial_i u)=0$ in $B_1\subset\R^2$ with $\sigma$ measurable and such that $\lambda\leq \sigma(x)\leq\Lambda$ for all $x\in B_1$, $0<\lambda\leq\Lambda$. Then there exists a universal constant $C>1$ such that
\[
\|u\|_{C^{0,\frac{4}{\pi}\arctan\left(\sqrt{\frac{\lambda}{\Lambda}}\right)}(B_\frac12)}\leq C\|u\|_{L^2(B_1)}.
\]
\end{thm}
We refer also to \cite{Alvino,Treskunov74} for similar results concerning nonhomogeneous equations.
\subsection{Nondivergence form equations}
In the case of nondivergence equations the following important result is due to Krylov-Safonov
\begin{thm}[Krylov-Safonov]
Let $u$ be a strong solution of $\mathrm{Tr}(A(x)D^2u)=0$ in $\R^n$ with $A$ measurable and such that $\lambda \mathbb{I}_n\leq A\leq\Lambda \mathbb{I}_n$, $0<\lambda\leq\Lambda$. Then there exist constants $\alpha\in(0,1)$ and $C>1$ universal such that
\[
\|u\|_{C^{0,\alpha}(B_\frac12)}\leq C\|u\|_{L^\infty(B_1)}.
\]
\end{thm}
This implies the H\"older regularity of gradients of solutions to \eqref{eqintro} in any dimension, since any directional derivative (or the incremental quotient) solves an equation in nondivergence form.\\
In the case of nondivergence equations in two variables the first results date back to \cite{Morrey,Nirenberg}. L. Nirenberg \cite{Nirenberg} proved the H\"older continuity of gradients of linear equations with an exponent $\alpha<\frac{\lambda}{4\pi\Lambda}$. An improved explicit H\"older exponent for the regularity of gradients appeared later in \cite{Talenti}, see also \cite[Theorem 12.4 and comments at p. 317]{GT} and \cite{Astalaetal}:
\begin{thm}[Talenti]
Let $u\in W^{2,2}_{\mathrm{loc}}$ be a strong solution of $\mathrm{Tr}(A(x)D^2u)=0$ in $\R^2$ with $A$ measurable and such that $\lambda \mathbb{I}_2\leq A\leq\Lambda \mathbb{I}_2$, $0<\lambda\leq\Lambda$. Then $u\in C^{1,\alpha}$, $\alpha\leq\frac{\lambda}{\Lambda}$, and there exists a universal $C>1$ such that
\[
\|u\|_{C^{1,\alpha}(B_\frac12)}\leq C\|u\|_{W^{1,p}(B_1)},\ p>2.
\]
\end{thm}
The H\"older exponent was later improved by Baernstein-Kovalev:
\begin{thm}[Baernstein-Kovalev]
Let $u\in W^{2,2}_{\mathrm{loc}}$ be a strong solution of $\mathrm{Tr}(A(x)D^2u)=0$ in $\R^2$ with $A$ measurable and such that $1/\sqrt{K} \mathbb{I}_2\leq A\leq\sqrt{K}\mathbb{I}_2$, $K\geq1$. Then $u\in C^{1,\tilde\alpha}_{\mathrm{loc}}$, where
\[
\tilde\alpha=\frac{1}{2\left(K+1\right)}\left(\sqrt{33+30K^{-1}+K^{-2}}-3-K^{-1}\right).
\]
\end{thm}

\section{$C^{2,\alpha}$ estimates via the variational regularity theory}\label{sec;div}
We have the following regularity estimate:
\begin{thm}\label{main1}
Let $u\in C^2(B_1)$ be a solution of $F(D^2u)=0$ in $B_1$. Then $u$ satisfies the regularity estimate
\[
\|u\|_{C^{2,\bar\alpha}(B_\frac12)}\leq C(\lambda,\Lambda)\|u\|_{L^\infty(B_1)}
\]
for an exponent $\bar\alpha$ given by
\[
\bar \alpha=\frac{2}{\pi}\frac{\frac{\lambda^3}{\Lambda^3}}{1+\frac{\lambda}{\Lambda}}\arctan\left(\frac{2\frac{\lambda}{\Lambda}}{1-\frac{\lambda}{\Lambda}}\right).
\]
\end{thm}
\begin{proof}
We start with the a priori bound, assuming $u\in C^4(B_1)$. Consider
\begin{equation}
F(D^2u)=0\text{ in }B_1\subset\R^2,
\end{equation}
with $F$ uniformly elliptic, namely for all $\xi\in\R^2$
\[
\lambda|\xi|^2\leq F_{ij}(M)\xi_i\xi_j\leq\Lambda|\xi|^2.
\]
Set $v=\partial_e u$, where $e\in\R^n$ such that $|e|=1$. Then $v$ solves
\[
\sum_{i,j=1}^2F_{ij}(D^2u)\partial_{ij}v=0,
\]
namely setting $a_{ij}=F_{ij}$, $v$ solves the PDE
\[
a_{11}(x)\partial_{11}v+2a_{12}(x)\partial_{12}v+a_{22}(x)\partial_{22}v=0
\]
Since $\lambda|\xi|^2\leq A\xi\cdot\xi\leq\Lambda|\xi|^2$, where $A(x)=\begin{pmatrix} a_{11} & a_{12}\\ a_{12} & a_{22} \end{pmatrix}$, the eigenvalues $\lambda_{1}^A\leq\lambda_{2}^A$ of $A$ satisfy
\[
\lambda\leq \lambda_{1}^A=\min_{|\xi|=1}A\xi\cdot\xi\leq Ae_1\cdot e_1=a_{11}\leq \max_{|\xi|=1}A\xi\cdot\xi=\lambda_2^A\leq\Lambda.
\]
This implies that
\[
\lambda a_{11}\leq \lambda_1^A\lambda_2^A=\mathrm{det}(A)=a_{11}a_{22}-a_{12}^2.
\]
Since $a_{11}\geq\lambda$, we have $\mathrm{det}(A)\geq\lambda^2$. Since $a_{22}\geq \lambda>0$ we can write
\[
\frac{a_{11}(x)}{a_{22}(x)}\partial_{11}v+\frac{2a_{12}(x)}{a_{22}(x)}\partial_{12}v+\partial_{22}v=0.
\]
We set 
\[
\tilde a(x):=\frac{a_{11}(x)}{a_{22}(x)}\text{ and }\tilde{b}(x):=\frac{2a_{12}(x)}{a_{22}(x)}.
\]
We differentiate with respect to $x_1$ and find for $w=\partial_1 v$
\begin{equation}\label{eqw}
\partial_1\left(\tilde a(x)\partial_{1}w+\tilde b(x)\partial_{2}w\right)+\partial_{22} w=0.
\end{equation}
This PDE can be written in divergence form as $\mathrm{div}(\tilde A(x)Dw)=0$ with
\[
\tilde A(x)=
\begin{pmatrix}
\tilde a(x) & \tilde b(x)\\
0 & 1
\end{pmatrix}.
\]
Since $\lambda\leq a_{11}\leq\Lambda$ and $\lambda\leq a_{22}\leq\Lambda$ we have for all $\xi\in\R^2$
\[
\frac{\lambda}{\Lambda}|\xi|^2\leq \tilde A_{ij}\xi_i\xi_j\leq \frac{\Lambda}{\lambda}|\xi|^2.
\]
We can write the previous matrix as
\[
\tilde A(x)=
\begin{pmatrix}
\tilde a(x) & \tilde b(x)/2\\
\tilde b(x)/2 & 1
\end{pmatrix}+
\begin{pmatrix}
 0 & \tilde b(x)/2\\
-\tilde b(x)/2 & 0
\end{pmatrix}=\tilde A^s+\tilde A^w.
\]
This matrix is of the same form as the one in Theorem \ref{treskunov} by taking $a(x)=\tilde a(x)$, $d(x)=\tilde b(x)/2$, $b(x)=\tilde b(x)/2$, $b(x)+d(x)=\tilde b(x)$, $b(x)-d(x)=0$, $c(x)=1$. Note that
\[
\frac{\lambda}{\Lambda}\leq \tilde a(x)\leq \frac{\Lambda}{\lambda}
\]
and
\[
d(x)=\frac{a_{12}}{a_{22}}.
\]
Since
\[
|a_{12}|\leq \frac{\Lambda-\lambda}{2},
\]
and $a_{22}\geq\lambda$, we have the inequality 
\[
|d(x)|\leq \frac{\Lambda-\lambda}{2\lambda}=:\tilde\eta
\]

First, note that $\mathrm{Tr}(\tilde A^s)>0$ and $\mathrm{det}(\tilde A^s)=\frac{\mathrm{det}(A)}{a_{22}^2}>0$, therefore it is positive definite. Moreover, $\tilde A^s$ satisfies
\[
 \theta\mathbb{I}_2\leq \tilde A^s\leq \Theta\mathbb{I}_2
\]
with
\[
\theta:=\frac{\lambda^3}{\Lambda^2(\Lambda+\lambda)}\text{ and }\Theta:=\frac{\Lambda^2(\Lambda+\lambda)}{\lambda^3}.
\]
In fact
\[
B_1=\begin{pmatrix}
\tilde a(x)-\theta & \tilde b(x)/2\\
\tilde b(x)/2 & 1-\theta
\end{pmatrix}
\]
is positive semidefinite. We have
\begin{multline*}
\mathrm{det}(B_1)=\tilde a(x)- \frac{\tilde b^2(x)}{4}+\theta^2-(\tilde a(x)+1)\theta=\frac{\mathrm{det}(A)}{a_{22}^2}+\theta^2-(\tilde a(x)+1)\theta\\
\geq \frac{\lambda^2}{\Lambda^2}+\theta^2-\left(\frac{\Lambda}{\lambda}+1\right)\frac{\lambda^3}{\Lambda^2(\Lambda+\lambda)}=\theta^2
\end{multline*}
 and $\mathrm{Tr}(B_1)\geq0$ ($\theta\leq1$ and $\tilde a-\theta\geq \frac{\lambda}{\Lambda}-\frac{\lambda^2}{\Lambda^2+\lambda\Lambda}\frac{\lambda}{\Lambda}$), while
\[
B_2=\begin{pmatrix}
\tilde a(x)-\Theta & \tilde b(x)/2\\
\tilde b(x)/2 & 1-\Theta
\end{pmatrix}
\]
is negative semidefinite, since $\mathrm{Tr}(B_2)\leq0$ ($\tilde a-\Theta\leq \frac{\Lambda}{\lambda}-\frac{\Lambda^2}{\lambda^2}\left(1+\frac{\Lambda}{\lambda}\right)$ and $1-\Theta=1-\frac{\Lambda^2}{\lambda^2}\left(1+\frac{\Lambda}{\lambda}\right)$)
\[
\mathrm{det}(B_2)=\frac{\mathrm{det}(A)}{a_{22}^2}+\Theta^2-(\tilde a(x)+1)\Theta\geq \frac{\lambda^2}{\Lambda^2}+\frac{\Lambda^2(\Lambda+\lambda)^2}{\lambda^4}\left(\frac{\Lambda^2}{\lambda^2}-1\right)\geq0.
\]
After normalizing the coefficient matrix $\tilde A$ and applying Theorem \ref{treskunov} with 
\[
\nu=\frac{\theta}{\Theta}=\frac{\frac{\lambda^6}{\Lambda^6}}{\left(1+\frac{\lambda}{\Lambda}\right)^2}<1\text{ and }\eta=\frac{\tilde\eta}{\Theta},
\]
 we conclude that $w\in C^{0,\bar \alpha}(B_\frac12)$ by Theorem \ref{treskunov}, cf. \cite{Treskunov74,Treskunov87}, with 
\[
\bar \alpha=\frac{2}{\pi}\sqrt{\nu}\arctan\left(\frac{\sqrt{\nu}}{\eta}\right)=\frac{2}{\pi}\frac{\frac{\lambda^3}{\Lambda^3}}{1+\frac{\lambda}{\Lambda}}\arctan\left(\frac{2\frac{\lambda}{\Lambda}}{1-\frac{\lambda}{\Lambda}}\right),
\]
and obtain the estimates
\[
\|w\|_{C^{\bar\alpha}(B_\frac12)}\leq C(\lambda,\Lambda)\|w\|_{L^2(B_1)}.
\]
Moreover, we have
\[
\|w\|_{C^{\bar\alpha}(B_\frac12)}\leq \tilde C(\lambda,\Lambda)\|w\|_{L^\infty(B_1)}
\]
for a different constant $\tilde C>0$. Since we can exchange the roles of $x_1$ and $x_2$ and the unitary direction $e$ is arbitrary, we can conclude that $u\in C^{2,\bar\alpha}$ in the interior of the domain. The dependence on $\|u\|_{L^\infty}$ follows by interpolation inequalities \cite[Lemma 6.35]{GT}, since for each $\eps>0$ there exists $C_\eps>0$ such that
\[
\|u\|_{C^{1,1}(B_\frac12)}\leq \eps\|u\|_{C^{2,\bar \alpha}(B_1)}+C_\eps\|u\|_{L^\infty(B_1)}.
\]
The estimate now follows applying the abstract result in \cite[Lemma 2.27]{FRRO}.\\
We now discuss the regularity estimate, starting from $C^2$ solutions. Note that $v(x)=\frac{u(x+h)-u(x)}{|h|}\in C^{2}(B_{1-|h|})$ for a small $h$. Since $F$ is translation invariant, we have that $F(D^2u(x+h))=0$ in $B_{1-|h|}$ by \cite[Section 5.3]{CC}. This implies that $v$ solves
\[
\mathrm{Tr}(A(x)D^2v):=\int_0^1F_{ij}(\theta D^2u(x+h)+(1-\theta)D^2u(x))\,d\theta\partial_{ij} v=0\text{ in }B_{1-|h|}.
\]
Since $F$ is uniformly elliptic, it is Lipschitz continuous in the matrix variable \cite{CC}, and $A$ is uniformly elliptic with the same ellipticity constants $\lambda,\Lambda$. Arguing as in \cite[Theorem 4.9]{FRRO}, we see that $\partial_1 v$ solves in weak sense an equation in divergence form with bounded and measurable coefficients given by the matrix $\tilde A$.
\end{proof}

\begin{rem}
One can also reach an explicit H\"older exponent, depending on $\lambda/\Lambda$, for divergence form equations via the Widman hole-filling technique \cite[Section 3]{MooneyBUMI} by taking care of the constants in the estimates.
\end{rem}

\begin{rem}
The idea that second derivatives of fully nonlinear equations solve linear elliptic PDEs without lower order terms is not new. See in particular \cite{MooneyCPAA,FRRO,GT} for a similar approach via equations in divergence form, \cite[Lemma 2.3]{NamSavin}, where the authors proved that mixed second derivatives solve a linear equation in nondivergence form, and \cite[Remark 2]{CaffarelliYuan}.
\end{rem}

It is worth remarking that the previous proof gives some new Bernstein-type classification results since the equation of second derivatives appears in divergence form. To deduce these weaker properties it is enough that the symmetric part of the matrix $\tilde A$ is uniformly elliptic, see \cite{MooneyCPAA}. For instance, one can prove that if $w$ is one-side bounded and solves $\mathrm{div}(\tilde A(x)Dw)\geq 0$ in $\R^2$ with $\tilde A$ having bounded coefficients and uniformly elliptic, then $w$ must be a constant. This result is known via Caccioppoli-type estimates using the test function $(k-u)^+\eta^2$ with $\eta$ cut-off of logarithmic type, cf. \cite{MooneyCPAA}. Another proof is the following: consider the vector field $Z=e^{w}\tilde A(x)Dw$ (in the case of supersolutions bounded from below take $Z=e^{-w}(-\tilde A(x)Dw)$). We have
\[
\int_{B_r}\mathrm{div}(Z)\,dx=\int_{B_r}e^w\tilde A(x)Dw\cdot Dw\,dx+\underbrace{\int_{B_r}e^w\mathrm{div}(\tilde A(x)Dw)\,dx}_{\geq0}\geq \frac{\lambda}{\Lambda}\int_{B_r}e^w|Dw|^2\,dx.
\]
On the other hand, since $A\in L^\infty$, we have by the divergence theorem
\[
\int_{B_r}\mathrm{div}(Z)\,dx\leq C\frac{\Lambda}{\lambda}\left(\int_{\partial B_r}e^w|Dw|^2\,dS\right)^\frac12\left(\int_{\partial B_r}e^w\,dS\right)^\frac12.
\]
After defining $H(r)=\int_{B_r}e^w|Dw|^2\,dx$ and noting that $H'(r)=\int_{\partial B_r}e^w|Dw|^2\,dS$, we get the inequality
\[
\frac{H'(r)}{H^2(r)}\geq\frac{1}{C^2} \frac{\lambda^4}{\Lambda^4}\frac{1}{\int_{\partial B_r}e^w\,dS}.
\]
Integrating in $[R,r]$ and using that $w\leq K$ we get
\[
\frac{1}{H(R)}\geq \frac{1}{C^2} \frac{\lambda^4}{\Lambda^4}e^{-K}\left(\log r-\log R\right),
\]
which gives the statement after sending $r\to+\infty$, since $D w=0$ and $w$ is constant in $\R^2$. If $F\in C^1$, this implies by a result of \cite{HuangProcAMS} that $u\in C^{2,\alpha}$ for any $\alpha<1$, with estimates depending on $F$ and the modulus of continuity of $DF$. We also note that the previous argument provides a new proof of the Liouville property for homogeneous equations in nondivergence form with bounded and measurable coefficients: $C^2$ solutions of $\mathrm{Tr}(A(x)D^2u)=0$ with bounded gradient in the plane are linear functions, cf. \cite{MooneyCPAA}. In fact, it is enough to divide the equation by $a_{22}(x)\geq\lambda$ and then differentiate with respect to $x_1$ to find that $w=\partial_1 u$ solves $\mathrm{div}(\tilde A(x)Dw)=0$ in $\R^2$.

\section{$C^{2,\alpha}$ estimates via the nonvariational regularity theory}\label{sec;nondiv}
We prove now better $C^{2,\alpha}$ quantitative estimates using the theory of nondivergence form equations in the plane.
\begin{thm}\label{main3}
Let $u\in C^2(B_1)$ be a solution of $F(D^2u)=0$ in $B_1$. Then $u$ satisfies the estimate
\[
\|u\|_{C^{2,\tilde\alpha}(B_\frac12)}\leq C(\lambda,\Lambda)\|u\|_{L^\infty(B_1)},
\]
where
\[
\tilde\alpha=\frac{\lambda}{2\left(\Lambda+\lambda\right)}\left(\sqrt{33+30\frac{\lambda}{\Lambda}+\frac{\lambda^2}{\Lambda^2}}-3-\frac{\lambda}{\Lambda}\right).
\]
\end{thm}
\begin{proof}
We assume $u\in C^3$, the case when $u\in C^2$ can be addressed using incremental quotients as discussed at the end of the proof of Theorem \ref{main1}. Consider
\begin{equation*}
F(D^2u)=0\text{ in }B_1\subset\R^2,
\end{equation*}
with $F$ uniformly elliptic with the same ellipticity constants, namely for $\xi\in\R^2$
\[
\lambda|\xi|^2\leq F_{ij}(M)\xi_i\xi_j\leq\Lambda|\xi|^2.
\]
Set $v=\partial_e u$, $e\in\R^n$ such that $|e|=1$. Then $v$ solves
\[
\sum_{i,j}F_{ij}(D^2u)\partial_{ij}v=0,
\]
that is, setting $a_{ij}=F_{ij}$, $v$ solves the PDE in nondivergence form
\[
a_{11}(x)\partial_{11}v+2a_{12}(x)\partial_{12}v+a_{22}(x)\partial_{22}v=0
\]
If we apply \cite{Talenti} we get $v\in C^{1,\frac{\lambda}{\Lambda}}$, which implies $u\in C^{2,\lambda/\Lambda}$.
If, instead, we normalize the coefficients and apply \cite[Theorem 1.1]{BaernsteinKovalev} we find
\[
v\in C^{1,\tilde\alpha},
\]
where
\[
\tilde\alpha=\frac{\lambda}{2\left(\Lambda+\lambda\right)}\left(\sqrt{33+30\frac{\lambda}{\Lambda}+\frac{\lambda^2}{\Lambda^2}}-3-\frac{\lambda}{\Lambda}\right).
\]
Therefore
\[
\|v\|_{C^{1,\tilde\alpha}(B_\frac12)}\leq C\|u\|_{W^{2,2}(B_1)}.
\]
This gives the statement by interpolation inequalities and since the direction $e$ is arbitrary, obtaining $u\in C^{2,\tilde\alpha}(B_\frac12)$.
\end{proof}
\begin{rem}
Note that
\[
\tilde\alpha=\frac{\frac{\lambda}{\Lambda}}{2\left(1+\frac{\lambda}{\Lambda}\right)}\left(\sqrt{33+30\frac{\lambda}{\Lambda}+\frac{\lambda^2}{\Lambda^2}}-3-\frac{\lambda}{\Lambda}\right)>\frac{\lambda}{\Lambda}.
\]
Therefore the previous result improves the exponent found via the theory of equations in divergence form in Theorem \ref{main1}. It is worth remarking that  G. Talenti proved a $C^{1,\alpha}$ estimate for nondivergence inhomogeneous equations in the plane of the form
\[
\mathrm{Tr}(A(x)D^2u)=g\text{ in }B_1\subset\R^2
\]
with $g\in L^q(B_1)$, $q>2$, and $\alpha<\min\{\lambda/\Lambda,1-2/q\}$. We can thus apply \cite{Talenti} to a inhomogeneous equation solved by the directional derivative $v=\partial_e u$ to find a $C^{2,\alpha}$, $\alpha<\min\{\lambda/\Lambda,1-2/q\}$, regularity result for the equation
\[
F(D^2u)=g\in W^{1,q},\ q>2.
\]
Other results for nondivergence equations appeared in \cite{TreskunovJMS}, where the author proved that the derivatives of solutions to homogeneous nondivergence equations in the plane satisfy the H\"older condition with
\[
\hat \alpha=\frac{\sqrt{33}-3}{2}\frac{\frac{\lambda}{\Lambda}}{\frac{\lambda^2}{\Lambda^2}+1}.
\]
When $\frac{\lambda}{\Lambda}\to0$ this exhibits the same behavior of the exponent found in \cite{BaernsteinKovalev}, since
\[
\hat\alpha\sim \frac{\sqrt{33}-3}{2}\frac{\lambda}{\Lambda}>\frac{\lambda}{\Lambda}.
\]
However, when $\frac{\lambda}{\Lambda}\to1$ we have $\hat\alpha\to \frac{\sqrt{33}-3}{4}<1$, while $\tilde\alpha\to1$. 
\end{rem}

\begin{rem}
The paper \cite{Astalaetal}, cf. Theorem 1.5 therein, provided a $W^{2,q}$ regularity result for nondivergence form equations $\mathrm{Tr}(A(x)D^2u)=0$ when $q<q(\lambda,\Lambda)$ under the assumption that $\mathrm{det}(A(x))=1$. This would imply a $W^{3,p}$ estimate for solutions of \eqref{eqintro} when $p<\frac{2\Lambda}{\Lambda-\lambda}$.
\end{rem}

\begin{rem}
The approach of Theorem \ref{main3} does not work in dimension $n\geq3$. In fact, $C^{1,\alpha}$ estimates for linear nondivergence equations are false, cf. \cite[Remark 3.6]{MooneyBUMI} and \cite{Safonov}.
\end{rem}


\end{document}